\documentclass[12pt]{amsart}
\usepackage{amsmath,amsthm,amsfonts,amssymb,mathrsfs}
\date{\today}

\usepackage{color}
\input xymatrix
\xyoption{all}

\usepackage{hyperref}

\newtheorem{theorem}{Theorem}%[section]

\newtheorem{proposition}[theorem]{Proposition}
\newtheorem{corollary}[theorem]{Corollary}
\newtheorem{lemma}[theorem]{Lemma}
\theoremstyle{definition}
\newtheorem{example}[theorem]{Example}%[section]
\begin{document}

\title[On semitopological $\alpha$-bicyclic monoid]{On semitopological $\alpha$-bicyclic monoid}

\author[S.~Bardyla]{Serhiy~Bardyla}
\address{Faculty of Mathematics, National University of Lviv,
Universytetska 1, Lviv, 79000, Ukraine}
\email{sbardyla@yahoo.com}

\keywords{topological inverse semigroup, topological semigroup, semitopological semigroup, $\alpha$-bicyclic semigroup}

\subjclass[2010]{Primary 20M18, 22A15. Secondary 55N07}

\begin{abstract}
In this paper we consider a semitopological $\alpha$-bicyclic monoid $\mathcal{B}_{\alpha}$ and prove that it is algebraically isomorphic to a semigroup of all order isomorphisms between the principal upper sets of the ordinal $\omega^{\alpha}$. We prove that for every ordinal $\alpha$ for every $(a,b)\in \mathcal{B_{\alpha}}$ if either $a$ or $b$ is a non-limit ordinal then $(a,b)$ is an isolated point in $\mathcal{B}_{\alpha}$. We show that for every ordinal $\alpha<\omega+1$ every  locally compact semigroup topology on $\mathcal{B}_{\alpha}$ is discrete. However, we construct an example of a non-discrete locally compact topology $\tau_{lc}$ on $\mathcal{B}_{\omega+1}$ such that $(\mathcal{B}_{\omega+1},\tau_{lc})$ is a topological inverse semigroup. This example shows that there is a gap in \cite[Theorem~2.9]{Hogan-1984}, where is stated that for every ordinal $\alpha$ there is only discrete locally compact inverse semigroup topology on $\mathcal{B_{\alpha}}$.
\end{abstract}

\maketitle
In this paper all topological spaces are assumed to be Hausdorff. By $\mathbb{N}$ we denote the set of all positive integers. A semigroup $S$ is called inverse if for every $x\in S$ there exists a unique $y\in S$ such that $xyx = x$ and $yxy = y$. Later such an element $y$ will be denoted by $x^{-1}$ and will be called the inverse of $x$. The map $inv: S\rightarrow S$ which assigns to every $s\in S$ its inverse is called inversion. By $\omega$ we denote the first infinite ordinal.
A {\it topological} ({\it inverse}) {\it semigroup} is a topological space together with a continuous semigroup operation (and an~inversion, respectively). Obviously, the inversion defined on a topological inverse semigroup is a homeomorphism. If $S$ is a~semigroup (an~inverse semigroup) and $\tau$ is a topology on $S$ such that $(S,\tau)$ is a topological (inverse) semigroup, then we shall call $\tau$ an (\emph{inverse}) \emph{semigroup}  \emph{topology} on $S$. A {\it semitopological semigroup} is a topological space together with a separately continuous semigroup operation.
Let $f$ be a map between two partial ordered sets $(A,\leq_{A})$ and $(B,\leq_{B})$, then we shall call $f$ monotone if for every $a,b\in A$ if $a\leq_{A} b$ then $f(a)\leq_{B} f(b)$. We shall call $f$ an order isomorphism if $f$ is a monotone bijection and its inverse map $f^{-1}$ is also monotone.

For a partially ordered set $(A,\leq)$, for an arbitrary $X\subset A$ and $x\in A$ we write:
\begin{itemize}
  \item[$1)$] ${\downarrow} X=\{y\in A: $ there exists $ x\in X$ such that $y\leq x\}$;
  \item[$2)$] ${\uparrow} X=\{y\in A: $ there exists $ x\in X$ such that $x\leq y\}$;
  \item[$3)$] ${\downarrow} x={\downarrow}\{x\}$;
  \item[$4)$] ${\uparrow} x={\uparrow}\{x\}$;
  \item[$5)$] $X$ is a lower set if $X={\downarrow} X$;
  \item[$6)$] $X$ is an upper set if $X={\uparrow} X$;
  \item[$7)$] $X$ is a principal lower set if $X={\downarrow} x$ for some $x\in X$;
 \item[$8)$] $X$ is a principal upper set if $X={\uparrow} x$ for some $x\in X$.
\end{itemize}

The bicyclic monoid ${\mathscr{B}}(p,q)$ is the semigroup with the identity $1$ generated by two elements $p$ and $q$ subjected only to the condition $pq=1$. The distinct elements of ${\mathscr{B}}(p,q)$ are exhibited in the following useful array
\begin{equation*}
\begin{array}{ccccc}
  1      & p      & p^2    & p^3    & \cdots \\
  q      & qp     & qp^2   & qp^3   & \cdots \\
  q^2    & q^2p   & q^2p^2 & q^2p^3 & \cdots \\
  q^3    & q^3p   & q^3p^2 & q^3p^3 & \cdots \\
  \vdots & \vdots & \vdots & \vdots & \ddots \\
\end{array}
\end{equation*}
and the semigroup operation on ${\mathscr{B}}(p,q)$ is defined as
follows:
\begin{equation*}
    q^kp^l\cdot q^mp^n=q^{k+m-\min\{l,m\}}p^{l+n-\min\{l,m\}}.
\end{equation*}
It is well known that the bicyclic monoid ${\mathscr{B}}(p,q)$ is a bisimple (and hence simple) combinatorial $E$-unitary inverse semigroup and every non-trivial congruence on ${\mathscr{B}}(p,q)$ is a group congruence \cite{Clifford-Preston-1961-1967}. Also, a classic Andersen Theorem states that \emph{a simple semigroup $S$ with an idempotent is completely simple if and only if $S$ does not contain an isomorphic copy of the bicyclic semigroup} (see \cite{Andersen-1952} and \cite[Theorem~2.54]{Clifford-Preston-1961-1967}).
Observe that the bicyclic monoid can be represented as a semigroup of isomorphisms between principal upper sets of partially ordered set $(\mathbb{N},\leq)$ (see \cite{Lawson-1998}).
The bicyclic semigroup admits only the discrete semigroup topology and if a topological semigroup S contains it
as a dense subsemigroup then ${\mathscr{B}}(p,q)$ is an open subset of S (see \cite{Eberhart-Selden-1969}). In \cite{Bertman-West-1976} and \cite{Fihel} this result was extended for the case of semitopological semigroups and generalized bicyclic semigroup respectively. The problem of embedding of the bicyclic monoid into compact-like topological semigroups was discussed in \cite{BanDimGut-2010}, \cite{BanDimGut-2009}, \cite{GutRep-2007}.
Also, in \cite{Eberhart-Selden-1969} was described the closure of the bicyclic monoid ${\mathscr{B}}(p,q)$ in a locally compact topological inverse semigroup. In \cite{Gut-2015} was proved that a Hausdorff locally compact semitopological bicyclic semigroup with adjoined zero $0$ is either compact or discrete.

Among the other natural generalizations of bicyclic semigroup the polycyclic monoids and $\alpha$-bicyclic monoids play the major role.

The polycyclic monoids were introduced in \cite{Nivat-Perrot-1970}. In \cite{BardGut-2016(1)} there was introduced a notion of the $\lambda$-polycyclic monoid, which is a generalization of the polycyclic monoid and it was proved therein that for every cardinal $\lambda>1$ any $\lambda$-polycyclic monoid $\mathcal{P_{\lambda}}$ can be represented as a subsemigroup of the semigroup of all order isomorphisms between principal lower sets of the $\lambda$-ary tree with adjoined zero. In \cite{Jones-2011} and \cite{Jones-Lawson-2014} were investigated algebraic properties of the polycyclic monoid. The paper \cite{Mesyan-Mitchell-Morayne-Peresse-2013} is devoted to topological properties of the graph inverse semigroups which are the generalization of polycyclic monoids and it is proved in it that for every finite graph $E$ every locally compact semigroup topology on the graph inverse semigroup over $E$ is discrete, which implies that for every positive integer $n$ every locally compact semigroup topology on the $n$-polycyclic monoid is discrete. Algebraic and topological properties of the $\lambda$-polycyclic monoid were studied in \cite{BardGut-2016(1)} and it was proved therein that for every non-zero cardinal $\lambda$ every locally compact semigroup topology on the $\lambda$-Polycyclic monoid is discrete. In \cite{BardGut-2016(2)} the authors investigated the closure of $\lambda$-polycyclic monoid in topological inverse semigroups.

However, in this paper we are mostly concerned on the $\alpha$-bicyclic monoid. This monoid was introduced in \cite{Hogan-1973}.
Let $\alpha$ be an arbitrary ordinal and $<$ be the usual order on $\alpha$ such that $a<b$ iff $a\in b$ for every $a,b\in \alpha$. For every $a,b\in \alpha$ we write $a\leq b$ iff either $a=b$ or $a\in b$. Clearly, $\leq$ is a partial order on $\alpha$. By $+$ we will denote the usual ordinal addition. An ordinal $\alpha$ is said to be prime if it cannot be represented as a sum of two ordinals which are contained in $\alpha$. For every ordinals $a,b$ such that $a>b$ we let $c=a-b$ if $a=b+c$. Clearly, for every ordinals $a>b$ there exists a unique ordinal $c$ such that $a=b+c$. For more about ordinals see \cite{Kunen}, \cite{SierpWacl-1965} or \cite{Weiss}.
By the $\alpha$-bicyclic monoid $\mathcal{B_{\alpha}}$ we mean the set $\omega^{\alpha}\times\omega^{\alpha}$ endowed with the following binary operation:

\begin{equation*}
    (a,b)\cdot (c,d)=
    \left\{
      \begin{array}{cl}
        (a+(c-b),d), & \hbox{if~~} b\leq c;\\
        (a,d+(b-c)),   & \hbox{if~~} b>c;
      \end{array}
    \right.
\end{equation*}

Later on we will write $(a,b)(c,d)$ instead of $(a,b)\cdot(c,d)$.

In \cite{Hogan-1973}, there were considered algebraic properties of bisimple semigroups with well-ordered idempotents. In \cite{Selden 1985}, a non-discrete inverse semigroup topology on $\mathcal{B}_{2}$ was constructed. In \cite{Hogan 1987}, there were investigated inverse semigroup topologies on $\mathcal{B}_{\alpha}$.

Observe that every upper set of arbitrary ordinal $\alpha$ is principal.

By $\mathscr{J}_{\omega^\alpha}^{\!\!\nearrow}$ we shall denote the semigroup of all order isomorphisms between the principal upper sets of the ordinal $\omega^{\alpha}$ endowed with multiplication of composition of partial maps.

Topological semigroups of partial monotone bijections of linearly ordered sets were investigated in \cite{Chuchman-Gutik}, \cite{GutRep2}, \cite{GutRep1}. In \cite{GutRep2} it was proved that every locally compact topology on the semigroup of all partial cofinite monotone injective transformations of the set of positive integers is discrete. In \cite{Chuchman-Gutik} the authors proved that every Baire topology on the semigroup of almost monotone injective co-finite partial selfmaps of positive integers is discrete. In \cite{GutRep1}, it was proved that every Baire topology on the semigroup of all monotone injective partial selfmaps of the set of integers having cofinite domain and image is discrete. We observe that in \cite{Chuchman-Gutik} and \cite{GutRep1}, non-discrete non-Baire inverse semigroup topologies on the corresponding semigroups are constructed.

In this paper we consider a semitopological $\alpha$-bicyclic monoid $\mathcal{B}_{\alpha}$ and prove that it is algebraically isomorphic to a semigroup of all order isomorphisms between the principal upper sets of the ordinal $\omega^{\alpha}$. We prove that for every ordinal $\alpha$ for every $(a,b)\in \mathcal{B_{\alpha}}$ if either $a$ or $b$ is a non-limit ordinal then $(a,b)$ is an isolated point in $\mathcal{B}_{\alpha}$. We show that for every ordinal $\alpha<\omega+1$ every  locally compact semigroup topology on $\mathcal{B}_{\alpha}$ is discrete. However, we construct an example of a non-discrete locally compact topology $\tau_{lc}$ on $\mathcal{B}_{\omega+1}$ such that $(\mathcal{B}_{\omega+1},\tau_{lc})$ is a topological inverse semigroup. This example shows that there is a gap in \cite[Theorem~2.9]{Hogan-1984}, where is stated that for every ordinal $\alpha$ there is only discrete locally compact inverse semigroup topology on $\mathcal{B_{\alpha}}$.

\begin{proposition}\label{l1} For every ordinal $\alpha$ the semigroup $\mathscr{J}_{\omega^\alpha}^{\!\!\nearrow}$ of all order isomorphisms between principle upper sets of the ordinal $\omega^{\alpha}$ is isomorphic to the $\alpha$-bicyclic monoid $\mathcal{B_{\alpha}}$.
\end{proposition}

\begin{proof}
Observe that every principal upper set of $\omega^{\alpha}$ is an interval $[a,\omega^{\alpha})$ for some $a\in\omega^{\alpha}$.
Define a map $h:\mathscr{J}_{\omega^\alpha}^{\!\!\nearrow}\rightarrow \mathcal{B_{\alpha}}$ in the following way:
for an arbitrary order isomorphism $f:[a,\omega^{\alpha})\rightarrow [b,\omega^{\alpha})$ put $h(f)= (a,b)$. Clearly, if there exists an order isomorphism between two intervals $[a,\omega^{\alpha})$ and $[b,\omega^{\alpha})$ then it is unique.
Thus the map $h$ is injective. Since $\omega^{\alpha}$ is a prime ordinal for every $a,b< \omega^{\alpha}$ the upper sets $[a,\omega^{\alpha})$ and $[b,\omega^{\alpha})$ of $\omega^{\alpha}$ are order isomorphic and hence the map $h$ is surjective. Simple verifications show that the map $h$ is a homomorphism.
\end{proof}

\begin{lemma}\label{l2} Let $(\mathcal{B_{\alpha}},\tau)$ be a semitopological semigroup. Then for every ordinal $a\in \omega^{\alpha}$ the elements $(0,a)$ and $(a,0)$ are isolated points in $(\mathcal{B_{\alpha}},\tau)$.
\end{lemma}

\begin{proof}
Observe that if $e$ is an idempotent of semitopological semigroup $S$ both $eS$ and $Se$ are retracts of the space $S$ and hence are closed subsets of $S$.
Since $\{(0,0)\}= \mathcal{B_{\alpha}}\setminus((1,1)\mathcal{B_{\alpha}}\cup \mathcal{B_{\alpha}}(1,1))$, $(0,0)$ is an isolated point in $(\mathcal{B_{\alpha}},\tau)$.

Since $(0,a)(a,0)= (0,0)$, the separate continuity of multiplication implies that there exists an open neighborhood  $V((0,a))$ such that $V((0,a))(a,0)=\{(0,0)\}$. Fix any point $(c,d)\in V((0,a))$. Then $(c,d)(a,0)= (0,0)$. Hence $d=a$ and $c=0$ which implies that $V((0,a))=\{(0,a)\}$. Thus $(0,a)$ is an isolated point in $(\mathcal{B_{\alpha}},\tau)$. The proof of the statement that $(a,0)$ is an isolated point is similar.
\end{proof}

\begin{lemma}\label{l3} Let $(\mathcal{B_{\alpha}},\tau)$ be a semitopological semigroup. Then for every element $(a,b)\in \mathcal{B_{\alpha}}$ there exists a clopen neighborhood $V((a,b))$ of $(a,b)$ such that the following conditions hold:
\begin{itemize}
  \item[$1)$] for every $(c,d)\in V((a,b))$ $c\leq a$ and $d\leq b$;
  \item[$2)$] for every $(c,d)\in V((a,b))$ $a=c$ if and only if $b=d$.
\end{itemize}
\end{lemma}

\begin{proof}
Put $V((a,b))=\{x\in S: (0,a)x=(0,b)\}$. Observe that, by Lemma \ref{l2}, $(0,b)$ is an isolated point of the space $\mathcal{B_{\alpha}}$. Since $(0,a)(a,b)=(0,b)$, the separate continuity of the multiplication in $\mathcal{B_{\alpha}}$ implies that $V(a,b)$ is a clopen neighborhood of the point $(a,b)$. Fix any element $(c,d)\in V((a,b))$. Then $(0,a)(c,d)=(0,b)$. Clearly, the above equality holds in the only case when $c\leq a$. Hence $(0,a)(c,d)=(0,d+(a-c))=(0,b)$ and since $d+(a-c)=b$ we get that $d\leq b$. Moreover, if $a=c$ then $(0,a)(c,d)=(0,d)=(0,b)$ and hence $b=d$.
\end{proof}

Lemma \ref{l3} implies the following corollary:
\begin{corollary}\label{c1} Let $(\mathcal{B_{\alpha}},\tau)$ be a semitopological semigroup. Then for every finite ordinals $n,m$ the element $(n,m)$ is an isolated point in the space $(\mathcal{B_{\alpha}},\tau)$.
\end{corollary}

\begin{lemma}\label{l4} Let $(\mathcal{B_{\alpha}},\tau)$ is a semitopological semigroup. For every distinct ordinals $a,b<\alpha$  $(\omega^{a},\omega^{b})$ is an isolated point in $(\mathcal{B_{\alpha}},\tau)$.
\end{lemma}

\begin{proof}
First we consider the case when $a<b$. Since $(0,\omega^{a})(\omega^{a},\omega^{b})=(0,\omega^{b})$, Lemma \ref{l2} and separate continuity of the semigroup operation in $(\mathcal{B_{\alpha}},\tau)$ imply that there exists an open neighborhood $V((\omega^{a},\omega^{b}))$ of $(\omega^{a},\omega^{b})$ in $(\mathcal{B_{\alpha}},\tau)$ such that $(0,\omega^{a})V((\omega^{a},\omega^{b}))=\{(0,\omega^{b})\}$ and for the neighborhood $V((\omega^{a},\omega^{b}))$ the conditions of Lemma \ref{l3} hold. Then $(0,\omega^{a})(c,d)=(0,d+(\omega^{a}-c))=(0,\omega^{b})$ for every $(c,d)\in V((a,b))$. Hence $d+(\omega^{a}-c)=d+\omega^{a}=\omega^{b}$, but this equation is true only if $\omega^{a}=\omega^{b}$, a contradiction.

In the case when $a>b$ the proof is similar.
\end{proof}

By \cite[Theorem 17]{Weiss} each ordinal $\alpha$ has Cantor's normal form, that is $\alpha=n_{1}\omega^{\beta_{1}}+n_{2}\omega^{\beta_{2}}+...+n_{k}\omega^{\beta_{k}}$, where $n_{i}$ are positive integers and $\beta_{1},\beta_{2},...,\beta_{3}$ is a decreasing sequence of ordinals.

\begin{lemma}\label{l5} Let $(\mathcal{B_{\alpha}},\tau)$ be a semitopological semigroup, $a=n_{1}\omega^{\beta_{1}}+n_{2}\omega^{\beta_{2}}+...+n_{k}\omega^{\beta_{k}}$, $b=m_{1}\omega^{\gamma_{1}}+m_{2}\omega^{\gamma_{2}}+...+m_{t}\omega^{\gamma_{t}}$ be Cantor's normal forms of the ordinals $a$ and $b$ respectively. If $(a,b)\in \mathcal{B_{\alpha}}$ and $(\omega^{\beta_{k}},\omega^{\gamma_{t}})$ is an isolated point in $(\mathcal{B_{\alpha}},\tau)$ then $(a,b)$ is an isolated point in the space $(\mathcal{B_{\alpha}},\tau)$.
\end{lemma}
\begin{proof}
Suppose that $(\omega^{\beta_{k}},\omega^{\gamma_{t}})$ is an isolated point in $(\mathcal{B_{\alpha}},\tau)$.
Put $a^{*}=n_{1}\omega^{\beta_{1}}+n_{2}\omega^{\beta_{2}}+...+(n_{k}-1)\omega^{\beta_{k}}$ and $b^{*}=m_{1}\omega^{\gamma_{1}}+n_{2}\omega^{\gamma_{2}}+...+(n_{t}-1)\omega^{\gamma_{t}}$. Then $(0,a^{*})(a,b)(b^{*},0)= (\omega^{\beta_{k}},\omega^{\gamma_{t}})$. The separate continuity of the multiplication in $(\mathcal{B_{\alpha}},\tau)$ implies that there exists an open neighborhood $V((a,b))$ of $(a,b)$ such that $(0,a^{*})V((a,b))(b^{*},0)= \{(\omega^{\beta_{k}},\omega^{\gamma_{t}})\}$. Fix any element $(c,d)\in V((a,b))$. Obviously, $(0,a^{*})(c,d)(b^{*},0)$ can be equal to $(\omega^{\beta_{k}},\omega^{\gamma_{t}})$ only if $a^{*}\leq c$ and $b^{*}\leq d$. Then $(0,a^{*})(c,d)(b^{*},0)=(c-a^{*},d-b^{*})=(\omega^{\beta_{k}},\omega^{\gamma_{t}})$, which implies that $c=a$ and $d=b$. Hence $V((a,b))=\{(a,b)\}$.
\end{proof}

\begin{proposition}\label{t1} Let $(\mathcal{B_{\alpha}},\tau)$ be a semitopological semigroup and $a$ be a non-limit ordinal. Then for every ordinal $b\in \omega^{\alpha}$ both $(a,b)$ and $(b,a)$ are isolated points in the space $(\mathcal{B_{\alpha}},\tau)$.
\end{proposition}

\begin{proof}
Corollary \ref{c1} and Lemma \ref{l5} imply that both points $(a,b)$ and $(b,a)$ are isolated in $(\mathcal{B_{\alpha}},\tau)$ provided that $b$ is a non-limit ordinal. Hence it is sufficient to consider the case when $b$ is a limit ordinal. Suppose to the contrary that $(a,b)$ is a non-isolated point in $(\mathcal{B_{\alpha}},\tau)$. Since $(a,b)(b,0)=(a,0)$, the separate continuity of multiplication in $(\mathcal{B_{\alpha}},\tau)$, Lemmas \ref{l2} and \ref{l3} imply that there exists an open neighborhood $V((a,b))$ of $(a,b)$ satisfying the conditions of Lemma \ref{l3} such that $V((a,b))(b,0)=\{(a,0)\}$. Fix an arbitrary element $(c,d)\in V(a,b)\setminus\{(a,b)\}$. Then $(c,d)(b,0)=(c+(b-d),0)=(a,0)$. Hence $a=c+(b-d)$, but since $b$ is a limit ordinal, $b-d$ is also a limit ordinal. Then $c+(b-d)$ is a limit ordinal which contradicts to the assumption that $a$ is a non-limit ordinal. The proof of the statement that $(b,a)$ is an isolated point in $(\mathcal{B_{\alpha}},\tau)$ is similar.
\end{proof}

\begin{theorem}\label{t2} For each $\alpha<\omega+1$ every locally compact topological $\alpha$-bicyclic semigroup $(\mathcal{B_{\alpha}},\tau)$ is discrete.
\end{theorem}
\begin{proof}
Lemmas \ref{l4} and \ref{l5} imply that if each idempotent $(\omega^{a},\omega^{a})$ of the $(\mathcal{B_{\alpha}},\tau)$ is an isolated point then $\tau$ is discrete.

It is obvious that the subset $\{(n,m):n,m<\omega\}\cup \{\omega,\omega\}$ with the semigroup operation induced from $\mathcal{B}_{\alpha}$ is isomorphic to the bicyclic semigroup with adjoined zero. Then by Lemma \ref{l3} and \cite[Corollary 1]{Gut-2015}, $(\omega,\omega)$ is an isolated point in $(\mathcal{B_{\alpha}},\tau)$.

Suppose $(\mathcal{B_{\alpha}},\tau)$ is a non-discrete semigroup. Let $m$ be the smallest positive integer such that $(\omega^{m},\omega^{m})$ is a non isolated idempotent of $(\mathcal{B_{\alpha}},\tau)$.
We remark that by our assumption Lemmas \ref{l4} and \ref{l5} imply that $\{(a,b): a,b<\omega^{m}\}$ is a discrete subsemigroup of $(\mathcal{B_{\alpha}},\tau)$ which is algebraically isomorphic to $\mathcal{B}_{m}$. By Lemma \ref{l3} there exists a clopen compact neighborhood $W((\omega^{m},\omega^{m}))$ of $(\omega^{m},\omega^{m})$ such that $W((\omega^{m},\omega^{m}))\subset \mathcal{B}_{m}\cup\{(\omega^{m},\omega^{m})\}$. The continuity of the multiplication in $(\mathcal{B_{\alpha}},\tau)$ implies that there exists a compact open neighborhood $V((\omega^{m},\omega^{m}))\subseteq W((\omega^{m},\omega^{m}))$ of $(\omega^{m},\omega^{m})$ such that $(0,0)\notin V^{2}((\omega^{m},\omega^{m}))$. Since $V((\omega^{m},\omega^{m}))$ is compact and $(\omega^{m},\omega^{m})$ is the only non-isolated point in $V((\omega^{m},\omega^{m}))$, we see that $(\omega^{m},\omega^{m})$ is a limit point of every infinite sequence $x_{n}\in V((\omega^{m},\omega^{m}))$ consisting of mutually distinct elements.

For an arbitrary element $(y,0)\in \mathcal{B}_{m}$ put
\begin{equation*}
X_{(y,0)}=\{(0,x): (y,0)(0,x)=(y,x)\in V\}.
 \end{equation*}
Suppose that there exists an element $(y,0)\in\mathcal{B}_{m}$ for which the set $X_{(y,0)}$ is infinite. Let $X_{(y,0)}=\{(0,x_{k}): k\in\mathbb{N}\}$ be an enumeration of the set $X_{(y,0)}$.
Observe that
\begin{equation*}
(\omega^{m},\omega^{m})=\lim_{k\rightarrow\infty}((y,0)(0,x_{k}))=\lim_{k\rightarrow\infty}(y,x_{k}).
 \end{equation*}
Then for every ordinal $z<\omega^{m}$
\begin{equation*}
\lim_{k\rightarrow\infty}(z,x_{k})=(\omega^{m},\omega^{m}),
 \end{equation*}
because
\begin{equation*}
(\omega^{m},\omega^{m})=(z,y)(\omega^{m},\omega^{m})=(z,y)\lim_{k\rightarrow\infty}(y,x_{k})=
\lim_{k\rightarrow\infty}((z,y)(y,x_{k}))=\lim_{k\rightarrow\infty}(z,x_{k})
 \end{equation*}
For every $k\in\mathbb{N}$ let $c_{k}$ be the smallest ordinal such that $(x_{k},c_{k})\in V((\omega^{m},\omega^{m}))$. Since  $(0,0)\notin V^{2}((\omega^{m},\omega^{m}))$, there exists $k_{0}\in\mathbb{N}$ such that $c_{k}\neq 0$ for every $k>k_{0}$. Observe that $(\omega^{m},\omega^{m})$ is a zero of $\mathcal{B}_{m}$. The continuity of the multiplication in $(\mathcal{B_{\alpha}},\tau)$ implies that there exists an open compact neighborhood $O((\omega^{m},\omega^{m}))\subseteq V((\omega^{m},\omega^{m}))$ such that
\begin{equation*}
O((\omega^{m},\omega^{m}))(1,0)\cup O((\omega^{m},\omega^{m}))(\omega,0)\cup..\cup\ O((\omega^{m},\omega^{m}))(\omega^{m-1},0)\subseteq V((\omega^{m},\omega^{m})).
 \end{equation*}
But then there exists an infinite discrete space $\{(x_{k},c_{k}): k\in\mathbb{N}\}\subset V((\omega^{m},\omega^{m}))\setminus O((\omega^{m},\omega^{m}))$, which contradicts the compactness of $V((\omega^{m},\omega^{m}))$.
The obtained contradiction implies that $X_{(y,0)}$ is a finite set for every element $(y,0)\in \mathcal{B}_{m}$.

Since $V((\omega^{m},\omega^{m}))$ is infinite, there exist an infinite set $A=\{(y_{n},0): n\in\mathbb{N}\}\subset  \mathcal{B}_{m}$ such that $X_{(y_{n},0)}\neq\emptyset$. For an arbitrary element $(y_{n},0)\in A$ by $(0,z_{y_{n}})$ we denote an element of $X_{(y_{n},0)}$ with the greatest second coordinate. Clearly that $\{(y_{n},0)(0,z_{y_{n}})=(y_{n},z_{y_{n}})\}$ is an infinite sequence of $V((\omega^{m},\omega^{m}))$. Then we have that
\begin{equation*}
 \lim_{n\rightarrow \infty}(y_{n},z_{y_{n}})= (\omega^{m},\omega^{m}).
\end{equation*}
However,
\begin{equation*}
\begin{split}&
(\omega^{m},\omega^{m})=(\omega^{m},\omega^{m})(0,\omega^{m-1})=\lim_{n\rightarrow \infty}(y_{n},z_{y_{n}})(0,\omega^{m-1})=\\
&=\lim_{n\rightarrow \infty}(y_{n},\omega^{m-1}+z_{y_{n}})\neq (\omega^{m},\omega^{m}),\\
\end{split}
 \end{equation*}
because $\omega^{m-1}+z_{y_{n}}>z_{y_{n}}$, which contradicts the continuity of the multiplication in $(\mathcal{B}_{\alpha},\tau)$. Hence $(\mathcal{B_{\alpha}},\tau)$ is a discrete semigroup.

\end{proof}

However, the following example shows that there exists a non-discrete locally compact topology $\tau_{lc}$ on the $\mathcal{B}_{\omega+1}$ such that $(\mathcal{B}_{\omega+1},\tau_{lc})$ is a topological inverse semigroup.

\begin{example}\label{e1}
We define the topology $\tau_{lc}$ in the following way: all points distinct from $(n\omega^{\omega},m\omega^{\omega})$ for some positive integers $n,m$ are isolated, and the family $\mathscr{B}((n\omega^{\omega},m\omega^{\omega}))=\{U_{k}((n\omega^{\omega},m\omega^{\omega})): k\in\mathbb{N}\}$ forms a base of the topology $\tau_{lc}$ at the point $(n\omega^{\omega},m\omega^{\omega})$, where
\begin{equation*}
U_{k}((n\omega^{\omega},m\omega^{\omega}))=\{((n-1)\omega^{\omega}+\omega^{t},(m-1)\omega^{\omega}+\omega^{t}): t>k\}\cup\{(n\omega^{\omega},m\omega^{\omega})\}.
\end{equation*}
Clearly, $\tau_{lc}$ is a Hausdorff topology on $\mathcal{B}_{\omega+1}$. Since every open basic neighborhood of an arbitrary non isolated point $(n\omega^{\omega},m\omega^{\omega})$ is compact, $\tau_{lc}$ is locally compact.

It is obvious that for every positive integer $k$
\begin{equation*}
inv(U_{k}(n\omega^{\omega},m\omega^{\omega}))=U_{k}(m\omega^{\omega},n\omega^{\omega}).
\end{equation*}
Hence the inversion is continuous in $(\mathcal{B}_{\omega+1},\tau_{lc})$.

Let $a<\omega^{\omega+1}$ be an arbitrary ordinal. Below it will be convenient to use the following trivial modification of the Cantor's normal form of the ordinal $a$:
\begin{equation*}
a=n_{1}\omega^{\omega}+n_{2}\omega^{t_{2}}+..+n_{p}\omega^{t_{p}},
\end{equation*}
where $n_{1}$ is a non negative integer, $n_{2},..,n_{p}$ are positive integers and $\omega,t_{2},t_{3},..,t_{p}$ is a decreasing sequence of ordinals.

It is sufficient to check the continuity of the multiplication at the point $((a,b),(c,d))\in\mathcal{B}_{\omega+1}\times\mathcal{B}_{\omega+1}$ when at least one of the points $(a,b)$ or $(c,d)$ is non-isolated in $(\mathcal{B}_{\omega+1},\tau_{lc})$. Hence there are three cases to consider:
\begin{itemize}
  \item[(1)] $(n_{1}\omega^{\omega},m_{1}\omega^{\omega})(n\omega^{\omega},m\omega^{\omega})$, where $n,m,n_{1},m_{1}\in\mathbb{N}$;
  \item[(2)] $(a,b)(n\omega^{\omega},m\omega^{\omega})$, where $(a,b)$ is an isolated point in $(\mathcal{B}_{\omega+1},\tau_{lc})$, $n,m\in\mathbb{N}$;
  \item[(3)] $(n\omega^{\omega},m\omega^{\omega})(a,b)$, where $(a,b)$ is an isolated point in $(\mathcal{B}_{\omega+1},\tau_{lc})$, $n,m\in\mathbb{N}$;

\end{itemize}
Suppose that case $(1)$ holds, then we obtain the multiplication of the form
\begin{equation*}
(n_{1}\omega^{\omega},m_{1}\omega^{\omega})(n\omega^{\omega},m\omega^{\omega}).
\end{equation*}
It has the following three subcases:
\begin{itemize}
  \item[(1.1)] if $m_{1}<n$ then $(n_{1}\omega^{\omega},m_{1}\omega^{\omega})(n\omega^{\omega},m\omega^{\omega})=((n_{1}+n-m_{1})\omega^{\omega},m\omega^{\omega})$;
  \item[(1.2)] if $m_{1}=n$ then $(n_{1}\omega^{\omega},m_{1}\omega^{\omega})(n\omega^{\omega},m\omega^{\omega})=(n_{1}\omega^{\omega},m\omega^{\omega})$;
  \item[(1.3)] if $m_{1}>n$ then $(n_{1}\omega^{\omega},m_{1}\omega^{\omega})(n\omega^{\omega},m\omega^{\omega})=(n_{1}\omega^{\omega},(m+m_{1}-n)\omega^{\omega})$.
\end{itemize}
Consider subcase $(1.1)$. Let $U_{k}(((n_{1}+n-m_{1})\omega^{\omega},m\omega^{\omega}))$ be a basic open neighborhood of $((n_{1}+n-m_{1})\omega^{\omega},m\omega^{\omega})$. Then we state that
 \begin{equation*}
U_{k}((n_{1}\omega^{\omega},m_{1}\omega^{\omega}))U_{k}((n\omega^{\omega},m\omega^{\omega}))\subseteq U_{k}(((n_{1}+n-m_{1})\omega^{\omega},m\omega^{\omega})).
 \end{equation*}
Indeed, fix any elements
 \begin{equation*}
((n_{1}-1)\omega^{\omega}+\omega^{t},(m_{1}-1)\omega^{\omega}+\omega^{t})\in U_{k}((n_{1}\omega^{\omega},m_{1}\omega^{\omega}))
 \end{equation*}
 and
 \begin{equation*}
((n-1)\omega^{\omega}+\omega^{p},(m-1)\omega^{\omega}+\omega^{p})\in U_{k}((n\omega^{\omega},m\omega^{\omega})).
 \end{equation*}
 Then
 \begin{equation*}
 \begin{split}
 & ((n_{1}-1)\omega^{\omega}+\omega^{t},(m_{1}-1)\omega^{\omega}+\omega^{t})((n-1)\omega^{\omega}+\omega^{p},(m-1)\omega^{\omega}+\omega^{p})=\\
& ((n_{1}-1)\omega^{\omega}+\omega^{t}+((n-1)\omega^{\omega}+\omega^{p}-((m_{1}-1)\omega^{\omega}+\omega^{t})),(m-1)\omega^{\omega}+\omega^{p})=\\
& ((n_{1}-1)\omega^{\omega}+\omega^{t}+((n-1-m_{1}+1)\omega^{\omega}+\omega^{p}),(m-1)\omega^{\omega}+\omega^{p})=\\
& ((n_{1}-1+n-1-m_{1}+1)\omega^{\omega}+\omega^{p},(m-1)\omega^{\omega}+\omega^{p})=\\
& ((n_{1}+n-m_{1}-1)\omega^{\omega}+\omega^{p},(m-1)\omega^{\omega}+\omega^{p})\in U_{k}(((n_{1}+n-m_{1})\omega^{\omega},m\omega^{\omega})).\\
\end{split}
\end{equation*}
Consider subcase $(1.2)$. Then $m_{1}=n$. Let $U_{k}((n_{1}\omega^{\omega},m\omega^{\omega}))$ be a basic open neighborhood of $(n_{1}\omega^{\omega},m\omega^{\omega})$. Then we state that
 \begin{equation*}
 U_{k}((n_{1}\omega^{\omega},n\omega^{\omega}))U_{k}((n\omega^{\omega},m\omega^{\omega}))\subseteq U_{k}((n_{1}\omega^{\omega},m\omega^{\omega})).
\end{equation*}
Indeed, fix any elements
 \begin{equation*}
((n_{1}-1)\omega^{\omega}+\omega^{t},(n-1)\omega^{\omega}+\omega^{t})\in U_{k}((n_{1}\omega^{\omega},n\omega^{\omega}))
 \end{equation*}
 and
  \begin{equation*}
 ((n-1)\omega^{\omega}+\omega^{p},(m-1)\omega^{\omega}+\omega^{p})\in U_{k}((n\omega^{\omega},m\omega^{\omega})).
\end{equation*}
If $p>t$ then
\begin{equation*}
 \begin{split}
& ((n_{1}-1)\omega^{\omega}+\omega^{t},(n-1)\omega^{\omega}+\omega^{t})((n-1)\omega^{\omega}+\omega^{p},(m-1)\omega^{\omega}+\omega^{p})=\\
& ((n_{1}-1)\omega^{\omega}+\omega^{t}+((n-1)\omega^{\omega}+\omega^{p}-((n-1)\omega^{\omega}+\omega^{t})),(m-1)\omega^{\omega}+\omega^{p})=\\
& ((n_{1}-1)\omega^{\omega}+\omega^{t}+(\omega^{p}-\omega^{t}),(m-1)\omega^{\omega}+\omega^{p})=\\
& ((n_{1}-1)\omega^{\omega}+\omega^{p},(m-1)\omega^{\omega}+\omega^{p})\in U_{k}((n_{1}\omega^{\omega},m\omega^{\omega})).\\
\end{split}
\end{equation*}
If $p=t$ then we have the following:
\begin{equation*}
 \begin{split}
& ((n_{1}-1)\omega^{\omega}+\omega^{p},(n-1)\omega^{\omega}+\omega^{p})((n-1)\omega^{\omega}+\omega^{p},(m-1)\omega^{\omega}+\omega^{p})=\\
& ((n_{1}-1)\omega^{\omega}+\omega^{p},(m-1)\omega^{\omega}+\omega^{p})\in U_{k}((n_{1}\omega^{\omega},m\omega^{\omega})).\\
\end{split}
\end{equation*}
If $p<t$ then
\begin{equation*}
 \begin{split}
& ((n_{1}-1)\omega^{\omega}+\omega^{t},(n-1)\omega^{\omega}+\omega^{t})((n-1)\omega^{\omega}+\omega^{p},(m-1)\omega^{\omega}+\omega^{p})=\\
& ((n_{1}-1)\omega^{\omega}+\omega^{t},(m-1)\omega^{\omega}+\omega^{p}+((n-1)\omega^{\omega}+\omega^{t}-((n-1)\omega^{\omega}+\omega^{p})))=\\
& ((n_{1}-1)\omega^{\omega}+\omega^{t},(m-1)\omega^{\omega}+\omega^{p}+(\omega^{t}-\omega^{p}))=\\
& ((n_{1}-1)\omega^{\omega}+\omega^{t},(m-1)\omega^{\omega}+\omega^{t})\in U_{k}((n_{1}\omega^{\omega},m\omega^{\omega})).\\
\end{split}
\end{equation*}
Consider subcase $(1.3)$. In this subcase we have that $m_{1}>n$.\\
Let $U_{k}((n_{1}\omega^{\omega},(m+m_{1}-n)\omega^{\omega}))$ be a basic open neighborhood of $(n_{1}\omega^{\omega},(m+m_{1}-n)\omega^{\omega})$.\\
Then we state that
\begin{equation*}
U_{k}((n_{1}\omega^{\omega},m_{1}\omega^{\omega}))U_{k}((n\omega^{\omega},m\omega^{\omega}))\subseteq U_{k}((n_{1}\omega^{\omega},(m+m_{1}-n)\omega^{\omega})).
\end{equation*}
Indeed, fix any elements
\begin{equation*}
((n_{1}-1)\omega^{\omega}+\omega^{t},(m_{1}-1)\omega^{\omega}+\omega^{t})\in U_{k}((n_{1}\omega^{\omega},m_{1}\omega^{\omega}))
\end{equation*}
and
\begin{equation*}
((n-1)\omega^{\omega}+\omega^{p},(m-1)\omega^{\omega}+\omega^{p})\in U_{k}((n\omega^{\omega},m\omega^{\omega})).
\end{equation*}
 Then
 \begin{equation*}
  \begin{split}
 & ((n_{1}-1)\omega^{\omega}+\omega^{t},(m_{1}-1)\omega^{\omega}+\omega^{t})((n-1)\omega^{\omega}+\omega^{p},(m-1)\omega^{\omega}+\omega^{p})=\\
& ((n_{1}-1)\omega^{\omega}+\omega^{t},(m-1)\omega^{\omega}+\omega^{p}+((m_{1}-1)\omega^{\omega}+\omega^{t}-((n-1)\omega^{\omega}+\omega^{p})))=\\
& ((n_{1}-1)\omega^{\omega}+\omega^{t},(m-1)\omega^{\omega}+\omega^{p}+((m_{1}-1-n+1)\omega^{\omega}+\omega^{t}))=\\
& ((n_{1}-1)\omega^{\omega}+\omega^{t},(m+m_{1}-n-1)\omega^{\omega}+\omega^{t})\in U_{k}((n_{1}\omega^{\omega},(m+m_{1}-n)\omega^{\omega})).\\
\end{split}
\end{equation*}
Hence the continuity of the multiplication in $(\mathcal{B}_{\omega+1},\tau_{lc})$ holds in case $(1)$.

Consider case $(2)$.
It has the following three subcases:
\begin{itemize}
  \item[$(2.1)$] $a\neq n_{1}\omega^{\omega}$ and $b\neq m_{1}\omega^{\omega}$;
  \item[$(2.2)$] $a\neq n_{1}\omega^{\omega}$ and $b=m_{1}\omega^{\omega}$;
  \item[$(2.3)$] $a=n_{1}\omega^{\omega}$ and $b\neq m_{1}\omega^{\omega}$.
\end{itemize}
Consider subcase $(2.1)$
Let $a=n_{1}\omega^{\omega}+n_{2}\omega^{t_{2}}+..+n_{p}\omega^{t_{p}}$ and\\ $b=m_{1}\omega^{\omega}+m_{2}\omega^{r_{2}}+..+m_{c}\omega^{r_{c}}$, (note that $n_{1}$ and $m_{1}$ could be equal to $0$).\\
Then we have the following two subcases:
\begin{itemize}
  \item[$(2.1.1)$] if $m_{1}<n$ then
  \begin{equation*}
  (n_{1}\omega^{\omega}+n_{2}\omega^{t_{2}}+..+n_{p}\omega^{t_{p}},m_{1}\omega^{\omega}+m_{2}\omega^{r_{2}}+..+m_{c}\omega^{r_{c}})(n\omega^{\omega},m\omega^{\omega})=
  ((n_{1}+n-m_{1})\omega^{\omega},m\omega^{\omega});
  \end{equation*}
  \item[$(2.1.2)$] if $m_{1}\geq n$ then
  \begin{equation*}
  \begin{split}
  & (n_{1}\omega^{\omega}+n_{2}\omega^{t_{2}}+..+n_{p}\omega^{t_{p}},m_{1}\omega^{\omega}+m_{2}\omega^{r_{2}}+..+m_{c}\omega^{r_{c}})(n\omega^{\omega},m\omega^{\omega})=\\
  & =(n_{1}\omega^{\omega}+n_{2}\omega^{t_{2}}+..+n_{p}\omega^{t_{p}},(m+m_{1}-n)\omega^{\omega}+m_{2}\omega^{r_{2}}+..+m_{c}\omega^{r_{c}});\\
  \end{split}
   \end{equation*}
\end{itemize}
Let us prove the continuity in subcase $(2.1.1)$. Let $U_{k}( ((n_{1}+n-m_{1})\omega^{\omega},m\omega^{\omega}))$ be a basic open neighborhood of $((n_{1}+n-m_{1})\omega^{\omega},m\omega^{\omega})$. Note that
\begin{equation*}
(n_{1}\omega^{\omega}+n_{2}\omega^{t_{2}}+..+n_{p}\omega^{t_{p}},m_{1}\omega^{\omega}+m_{2}\omega^{r_{2}}+..+m_{c}\omega^{r_{c}})
\end{equation*}
is an isolated point in $(\mathcal{B}_{\omega+1},\tau_{lc})$. Then we state that
\begin{equation*}
 \begin{split}
 & (n_{1}\omega^{\omega}+n_{2}\omega^{t_{2}}+..+n_{p}\omega^{t_{p}},m_{1}\omega^{\omega}+m_{2}\omega^{r_{2}}+..+m_{c}\omega^{r_{c}})U_{r_{2}+t_{2}+k}((n\omega^{\omega},m\omega^{\omega}))\subseteq \\
  &
  \subseteq U_{k}(((n_{1}+n-m_{1})\omega^{\omega},m\omega^{\omega})). \\
  \end{split}
\end{equation*}
 Indeed, fix any element
 \begin{equation*}
 ((n-1)\omega^{\omega}+\omega^{t},(m-1)\omega^{\omega}+\omega^{t})\in U_{r_{2}+t_{2}+k}((n\omega^{\omega},m\omega^{\omega})).
 \end{equation*}
  Then
\begin{equation*}
\begin{split}
&
(n_{1}\omega^{\omega}+n_{2}\omega^{t_{2}}+..+n_{p}\omega^{t_{p}},m_{1}\omega^{\omega}+m_{2}\omega^{r_{2}}+..+m_{c}\omega^{r_{c}})((n-1)\omega^{\omega}+\omega^{t},(m-1)\omega^{\omega}+\\
&+\omega^{t})= (n_{1}\omega^{\omega}+n_{2}\omega^{t_{2}}+..+n_{p}\omega^{t_{p}}+((n-1)\omega^{\omega}+\omega^{t}-(m_{1}\omega^{\omega}+m_{2}\omega^{r_{2}}+..\\
&+m_{c}\omega^{r_{c}})),(m-1)\omega^{\omega}+\omega^{t})=\\
&=(n_{1}\omega^{\omega}+n_{2}\omega^{t_{2}}+..+n_{p}\omega^{t_{p}}+((n-1-m_{1})\omega^{\omega}+\omega^{t},(m-1)\omega^{\omega}+\omega^{t})=
\\
&
=((n_{1}+n-m_{1}-1)\omega^{\omega}+\omega^{t},(m-1)\omega^{\omega}+\omega^{t})\in U_{k}(((n_{1}+n-m_{1})\omega^{\omega},m\omega^{\omega})).
\\
\end{split}
\end{equation*}
Hence the continuity of the multiplication in subcase $(2.1.1)$ holds.

Consider subcase $(2.1.2)$. Note that both $(n_{1}\omega^{\omega}+n_{2}\omega^{t_{2}}+..+n_{p}\omega^{t_{p}},m_{1}\omega^{\omega}+m_{2}\omega^{r_{2}}+..+m_{c}\omega^{r_{c}})$
and $(n_{1}\omega^{\omega}+n_{2}\omega^{t_{2}}+..+n_{p}\omega^{t_{p}},(m+m_{1}-n)\omega^{\omega}+m_{2}\omega^{r_{2}}+..+m_{c}\omega^{r_{c}})$  are isolated points in the $(\mathcal{B}_{\omega+1},\tau_{lc})$. Then we state that
\begin{equation*}
 \begin{split}
 & (n_{1}\omega^{\omega}+n_{2}\omega^{t_{2}}+..+n_{p}\omega^{t_{p}},m_{1}\omega^{\omega}+m_{2}\omega^{r_{2}}+..+m_{c}\omega^{r_{c}})U_{0}((n\omega^{\omega},m\omega^{\omega}))= \\
  &
  =\{(n_{1}\omega^{\omega}+n_{2}\omega^{t_{2}}+..+n_{p}\omega^{t_{p}},(m+m_{1}-n)\omega^{\omega}+m_{2}\omega^{r_{2}}+..+m_{c}\omega^{r_{c}})\}. \\
  \end{split}
\end{equation*}
Indeed, fix any element
\begin{equation*}
((n-1)\omega^{\omega}+\omega^{t},(m-1)\omega^{\omega}+\omega^{t})\in U_{0}((n\omega^{\omega},m\omega^{\omega})).
\end{equation*}
 Then
\begin{equation*}
\begin{split}
&
(n_{1}\omega^{\omega}+n_{2}\omega^{t_{2}}+..+n_{p}\omega^{t_{p}},m_{1}\omega^{\omega}+m_{2}\omega^{r_{2}}+..+m_{c}\omega^{r_{c}})((n-1)\omega^{\omega}+\omega^{t},(m-1)\omega^{\omega}+\\
&+\omega^{t})=(n_{1}\omega^{\omega}+n_{2}\omega^{t_{2}}+..+n_{p}\omega^{t_{p}},(m-1)\omega^{\omega}+\omega^{t}+(m_{1}\omega^{\omega}+m_{2}\omega^{r_{2}}+..+m_{c}\omega^{r_{c}}-\\
&-((n-1)\omega^{\omega}+\omega^{t})))=\\
&(n_{1}\omega^{\omega}+n_{2}\omega^{t_{2}}+..+n_{p}\omega^{t_{p}},(m-1)\omega^{\omega}+\omega^{t}+((m_{1}-n+1)\omega^{\omega}+m_{2}\omega^{r_{2}}+..+m_{c}\omega^{r_{c}})=\\
&=(n_{1}\omega^{\omega}+n_{2}\omega^{t_{2}}+..+n_{p}\omega^{t_{p}},(m+m_{1}-n)\omega^{\omega}+m_{2}\omega^{r_{2}}+..+m_{c}\omega^{r_{c}}).
\\
\end{split}
\end{equation*}
Hence the continuity of the multiplication in subcase $(2.1)$ holds.

Consider subcase $(2.2)$. Let $a=n_{1}\omega^{\omega}+n_{2}\omega^{t_{2}}+..+n_{p}\omega^{t_{p}}$ and $b=m_{1}\omega^{\omega}$. Then we have the following two subcases:
\begin{itemize}
  \item[$(2.2.1)$] if $m_{1}<n$ then
  \begin{equation*}
  (n_{1}\omega^{\omega}+n_{2}\omega^{t_{2}}+..+n_{p}\omega^{t_{p}},m_{1}\omega^{\omega})(n\omega^{\omega},m\omega^{\omega})=
  ((n_{1}+n-m_{1})\omega^{\omega},m\omega^{\omega});
  \end{equation*}
  \item[$(2.2.2)$] if $m_{1}\geq n$ then
  \begin{equation*}
  \begin{split}
  & (n_{1}\omega^{\omega}+n_{2}\omega^{t_{2}}+..+n_{p}\omega^{t_{p}},m_{1}\omega^{\omega})(n\omega^{\omega},m\omega^{\omega})=\\
  & =(n_{1}\omega^{\omega}+n_{2}\omega^{t_{2}}+..+n_{p}\omega^{t_{p}},(m+m_{1}-n)\omega^{\omega});\\
  \end{split}
   \end{equation*}
\end{itemize}

Consider subcase $(2.2.1)$. Let $U_{k}( ((n_{1}+n-m_{1})\omega^{\omega},m\omega^{\omega}))$ be a basic open neighborhood of $ ((n_{1}+n-m_{1})\omega^{\omega},m\omega^{\omega})$. Note that $(n_{1}\omega^{\omega}+n_{2}\omega^{t_{2}}+..+n_{p}\omega^{t_{p}},m_{1}\omega^{\omega})$  is an isolated point in the  $(\mathcal{B}_{\omega+1},\tau_{lc})$. Then we state that
\begin{equation*}
 \begin{split}
 & (n_{1}\omega^{\omega}+n_{2}\omega^{t_{2}}+..+n_{p}\omega^{t_{p}},m_{1}\omega^{\omega})U_{t_{2}+k}((n\omega^{\omega},m\omega^{\omega}))\subseteq \\
  &
  \subseteq U_{k}(((n_{1}+n-m_{1})\omega^{\omega},m\omega^{\omega})). \\
  \end{split}
\end{equation*}
 Indeed, fix any element
 \begin{equation*}
 ((n-1)\omega^{\omega}+\omega^{t},(m-1)\omega^{\omega}+\omega^{t})\in U_{t_{2}+k}((n\omega^{\omega},m\omega^{\omega})).
 \end{equation*}
  Then
\begin{equation*}
\begin{split}
&
(n_{1}\omega^{\omega}+n_{2}\omega^{t_{2}}+..+n_{p}\omega^{t_{p}},m_{1}\omega^{\omega})((n-1)\omega^{\omega}+\omega^{t},(m-1)\omega^{\omega}+\omega^{t})=
 \\
&
=(n_{1}\omega^{\omega}+n_{2}\omega^{t_{2}}+..+n_{p}\omega^{t_{p}}+((n-1)\omega^{\omega}+\omega^{t}-m_{1}\omega^{\omega}),(m-1)\omega^{\omega}+\omega^{t})=
\\
&
=(n_{1}\omega^{\omega}+n_{2}\omega^{t_{2}}+..+n_{p}\omega^{t_{p}}+((n-1-m_{1})\omega^{\omega}+\omega^{t}),(m-1)\omega^{\omega}+\omega^{t})=
\\
&
=((n_{1}+n-m_{1}-1)\omega^{\omega}+\omega^{t},(m-1)\omega^{\omega}+\omega^{t})\in U_{k}(((n_{1}+n-m_{1})\omega^{\omega},m\omega^{\omega})).
\\
\end{split}
\end{equation*}
Hence the continuity of the multiplication in subcase (2.2.1) holds.

Consider subcase $(2.2.2)$. Note that both $(n_{1}\omega^{\omega}+n_{2}\omega^{t_{2}}+..+n_{p}\omega^{t_{p}},m_{1}\omega^{\omega})$
and $(n_{1}\omega^{\omega}+n_{2}\omega^{t_{2}}+..+n_{p}\omega^{t_{p}},(m+m_{1}-n)\omega^{\omega})$ are isolated points in the $(\mathcal{B}_{\omega+1},\tau_{lc})$. Then we state that
\begin{equation*}
 \begin{split}
 & (n_{1}\omega^{\omega}+n_{2}\omega^{t_{2}}+..+n_{p}\omega^{t_{p}},m_{1}\omega^{\omega})U_{0}((n\omega^{\omega},m\omega^{\omega}))= \\
  &
  =\{(n_{1}\omega^{\omega}+n_{2}\omega^{t_{2}}+..+n_{p}\omega^{t_{p}},(m+m_{1}-n)\omega^{\omega})\}. \\
  \end{split}
\end{equation*}
Indeed, fix any element
\begin{equation*}
((n-1)\omega^{\omega}+\omega^{t},(m-1)\omega^{\omega}+\omega^{t})\in U_{0}((n\omega^{\omega},m\omega^{\omega})).
\end{equation*}
Then
\begin{equation*}
\begin{split}
&
(n_{1}\omega^{\omega}+n_{2}\omega^{t_{2}}+..+n_{p}\omega^{t_{p}},m_{1}\omega^{\omega})((n-1)\omega^{\omega}+\omega^{t},(m-1)\omega^{\omega}+\omega^{t})=
 \\
&
=(n_{1}\omega^{\omega}+n_{2}\omega^{t_{2}}+..+n_{p}\omega^{t_{p}},(m-1)\omega^{\omega}+\omega^{t}+(m_{1}\omega^{\omega}-((n-1)\omega^{\omega}+\omega^{t})))=
\\
&
=(n_{1}\omega^{\omega}+n_{2}\omega^{t_{2}}+..+n_{p}\omega^{t_{p}},(m-1)\omega^{\omega}+\omega^{t}+((m_{1}-n+1)\omega^{\omega})=
\\
&
=(n_{1}\omega^{\omega}+n_{2}\omega^{t_{2}}+..+n_{p}\omega^{t_{p}},(m+m_{1}-n)\omega^{\omega}).
\\
\end{split}
\end{equation*}
Hence the continuity of the multiplication in subcase $(2.2)$ holds.

Consider subcase $(2.3)$. Let $a=n_{1}\omega^{\omega}$ and $b=m_{1}\omega^{\omega}+m_{2}\omega^{r_{2}}+..+m_{c}\omega^{r_{c}}$.

Then we have the following two subcases:
\begin{itemize}
  \item[$(2.3.1)$] if $m_{1}<n$ then
  \begin{equation*}
  (n_{1}\omega^{\omega},m_{1}\omega^{\omega}+m_{2}\omega^{r_{2}}+..+m_{c}\omega^{r_{c}})(n\omega^{\omega},m\omega^{\omega})=
  ((n_{1}+n-m_{1})\omega^{\omega},m\omega^{\omega});
  \end{equation*}
  \item[$(2.3.2)$] if $m_{1}\geq n$ then
  \begin{equation*}
  \begin{split}
  & (n_{1}\omega^{\omega},m_{1}\omega^{\omega}+m_{2}\omega^{r_{2}}+..+m_{c}\omega^{r_{c}})(n\omega^{\omega},m\omega^{\omega})=\\
  & =(n_{1}\omega^{\omega},(m+m_{1}-n)\omega^{\omega}+m_{2}\omega^{r_{2}}+..+m_{c}\omega^{r_{c}});\\
  \end{split}
   \end{equation*}
\end{itemize}
Consider subcase $(2.3.1)$. Let $U_{k}( ((n_{1}+n-m_{1})\omega^{\omega},m\omega^{\omega}))$ be a basic open neighborhood of $((n_{1}+n-m_{1})\omega^{\omega},m\omega^{\omega})$. Then we state that
\begin{equation*}
 \begin{split}
 & (n_{1}\omega^{\omega},m_{1}\omega^{\omega}+m_{2}\omega^{r_{2}}+..+m_{c}\omega^{r_{c}})U_{r_{2}+k}((n\omega^{\omega},m\omega^{\omega}))\subseteq \\
  &
  \subseteq U_{k}(((n_{1}+n-m_{1})\omega^{\omega},m\omega^{\omega})). \\
  \end{split}
\end{equation*}
 Indeed, fix any element
 \begin{equation*}
 ((n-1)\omega^{\omega}+\omega^{t},(m-1)\omega^{\omega}+\omega^{t})\in U_{r_{2}+k}((n\omega^{\omega},m\omega^{\omega})).
\end{equation*}
  Then
\begin{equation*}
\begin{split}
&
(n_{1}\omega^{\omega},m_{1}\omega^{\omega}+m_{2}\omega^{r_{2}}+..+m_{c}\omega^{r_{c}})((n-1)\omega^{\omega}+\omega^{t},(m-1)\omega^{\omega}+\omega^{t})=
 \\
&
=(n_{1}\omega^{\omega}+((n-1)\omega^{\omega}+\omega^{t}-(m_{1}\omega^{\omega}+m_{2}\omega^{r_{2}}+..+m_{c}\omega^{r_{c}})),(m-1)\omega^{\omega}+\omega^{t})=
\\
&
=(n_{1}\omega^{\omega}+((n-1-m_{1})\omega^{\omega}+\omega^{t}),(m-1)\omega^{\omega}+\omega^{t})=
\\
&
=((n_{1}+n-m_{1}-1)\omega^{\omega}+\omega^{t},(m-1)\omega^{\omega}+\omega^{t})\in U_{k}(((n_{1}+n-m_{1})\omega^{\omega},m\omega^{\omega})).
\\
\end{split}
\end{equation*}
Hence the continuity of the multiplication in subcase $(2.3.1)$ holds.

Now let's consider the subcase $(2.3.2)$. Note that both $(n_{1}\omega^{\omega},m_{1}\omega^{\omega}+m_{2}\omega^{r_{2}}+..\\
+m_{c}\omega^{r_{c}})$ and $(n_{1}\omega^{\omega},(m+m_{1}-n)\omega^{\omega}+m_{2}\omega^{r_{2}}+..+m_{c}\omega^{r_{c}})$  are isolated points in  the $(\mathcal{B}_{\omega+1},\tau_{lc})$. Then we state that
\begin{equation*}
 \begin{split}
 & (n_{1}\omega^{\omega},m_{1}\omega^{\omega}+m_{2}\omega^{r_{2}}+..+m_{c}\omega^{r_{c}})U_{0}((n\omega^{\omega},m\omega^{\omega}))= \\
  &
  =\{(n_{1}\omega^{\omega},(m+m_{1}-n)\omega^{\omega}+m_{2}\omega^{r_{2}}+..+m_{c}\omega^{r_{c}})\}. \\
  \end{split}
\end{equation*}
Indeed, fix any element
\begin{equation*}
((n-1)\omega^{\omega}+\omega^{t},(m-1)\omega^{\omega}+\omega^{t})\in U_{0}((n\omega^{\omega},m\omega^{\omega})).
\end{equation*}
 Then
\begin{equation*}
\begin{split}
&
(n_{1}\omega^{\omega},m_{1}\omega^{\omega}+m_{2}\omega^{r_{2}}+..+m_{c}\omega^{r_{c}})((n-1)\omega^{\omega}+\omega^{t},(m-1)\omega^{\omega}+\omega^{t})=
 \\
&
=(n_{1}\omega^{\omega},(m-1)\omega^{\omega}+\omega^{t}+(m_{1}\omega^{\omega}+m_{2}\omega^{r_{2}}+..+m_{c}\omega^{r_{c}}-((n-1)\omega^{\omega}+\omega^{t})))=
\\
&
=(n_{1}\omega^{\omega},(m-1)\omega^{\omega}+\omega^{t}+((m_{1}-n+1)\omega^{\omega}+m_{2}\omega^{r_{2}}+...+m_{c}\omega^{r_{c}})=
\\
&
=(n_{1}\omega^{\omega},(m+m_{1}-n)\omega^{\omega}+m_{2}\omega^{r_{2}}+..+m_{c}\omega^{r_{c}}).
\\
\end{split}
\end{equation*}
Hence continuity of the multiplication holds in case $(2)$.

Since the inversion is continuous in $(\mathcal{B}_{\omega+1},\tau_{lc})$ and
\begin{equation*}
((a,b)(n\omega^{\omega},m\omega^{\omega}))^{-1}=(m\omega^{\omega},n\omega^{\omega})(b,a)
 \end{equation*}
case $(2)$ implies that the semigroup operation in the case $(3)$ is continuous as well.

Hence $(\mathcal{B}_{\omega+1},\tau_{lc})$ is a topological inverse semigroup.
\end{example}

\section*{Acknowledgements}

The author acknowledges Oleg Gutik for his comments and suggestions.

\end{document}